\documentclass[12pt,reqno]{amsart}
\usepackage{amssymb,amsmath,graphicx,
	amsfonts,euscript}
\usepackage{color}
\usepackage{mathrsfs}

\newcommand{\p}{{\mathbb P}}

\def\u{{u_{\nu}}}
\def\v{\bar{u}}

\def\w{\bar{\tau}}
\newcommand{\R}{{\mathbb R}}
\def\div{ \hbox{\rm div}\,  }
\newcommand\Z{{\mathbb{Z}}}
\def\La{\Lambda}
\def\aa{\phi}
\def\ddj{\dot \Delta_j}
\def\f{\frac}
\theoremstyle{plain}

\newtheorem{theorem}{Theorem}[section]
\newtheorem{lemma}[theorem]{Lemma}
\newtheorem{definition}[theorem]{Definition}

\newtheorem{remark}[theorem]{Remark}
\numberwithin{equation}{section}

\begin{document}

\title[Global well-posedness  and inviscid limits of the Oldroyd models]{Global  well-posedness  and inviscid limits of the generalized Oldroyd type models}

\author[Xiaoping Zhai, Yuanyuan Dan,  Yongsheng Li]{Xiaoping Zhai,  Yuanyuan Dan,  Yongsheng Li}

\address[X. Zhai]{ School of Mathematics and Statistics, Shenzhen University, Shenzhen, 518060, China.} \email{zhaixp@szu.edu.cn }

\address[Y. Dan]{Big data and Educational Statistics Application Laboratory, School of Statistics and Mathematics, Guangdong University of Finance and Economics,
Guangzhou, 510320, China.} \email{danyy2014@126.com (Corresponding Author)}

\address[Y. Li]{ School of Mathematics, South China University of Technology, Guangzhou, 510640, China.} \email{yshli@scut.edu.cn}

\begin{abstract}
We  obtain  the global small solutions to the generalized  Oldroyd-B model   without damping on the stress tensor in $\R^n$.
Our result  give positive answers partially to the question proposed by Elgindi and  Liu (Remark 2 in Elgindi and  Liu [J Differ
Equ 259:1958--1966, 2015)].
The proof relies heavily on the trick of transferring dissipation from $u$ to $\tau$, and a new commutator estimate which may be of interest for future works.
Moreover, we prove a global result of inviscid limit of two dimensional Oldroyd type models in the Sobolev spaces.
The convergence rate is also obtained simultaneously.
	\end{abstract}
\maketitle

\section{ Introduction and the main results}
The Oldroyd-B model is a typical prototypical model for viscoelastic fluids, which describes the motion of some viscoelastic flows:
\begin{eqnarray}\label{mmm}
\left\{\begin{aligned}
&u_t + u\cdot \nabla u -  \nu \Delta u + \nabla \pi =  K_1 \div \tau,  \\
&\tau_t + u\cdot \nabla \tau   + \beta \tau + Q(\tau, \nabla u) =  K_2 D (u),\\
&\div u = 0,
\end{aligned}\right.
\end{eqnarray}
where  $\nu> 0$, $\beta\ge0$, $K_1\ge 0, K_2\ge 0$, $u$ stands for the velocity of the
fluid
and $\pi$ for the pressure and $\tau$ is a symmetric tensor. 
 The parameters
$\nu $, $K_2$ correspond respectively to $\frac{\theta}{\mathrm{Re}}$ and $\frac{2(1-\theta)}{\mathrm{Re} \ \mathrm{We}}$, where $\mathrm{Re}$ is the Reynolds number, $\theta$ is
the ratio between the so-called relaxation and retardation times and $\mathrm{We}$ is the Weissenberg
number which measures the elasticity of the fluid.
We denote by
$\Omega(u)$  the vorticity tensor  and $D(u)$  the deformation tensor, namely
$$
\Omega(u) = \frac{1}{2} \big( \nabla u - (\nabla u)^{T} \big),\quad D(u) = \frac{1}{2} \big( \nabla u + (\nabla u)^{T} \big).
$$
$ Q(\tau, \nabla u)$ is a given bilinear form which can be chosen as
$$
Q(\tau, \nabla u)= \tau \Omega(u) - \Omega(u) \tau + b(D(u) \tau + \tau D(u)),
$$
$b$ is a parameter in $[-1,1]$.

The Oldroyd-B model is a classical model for dilute solutions of polymers suspended in a
viscous incompressible solvent (see \cite{bird}).
 The system \eqref{mmm} describes the motion of the incompressible
fluid satisfying the Oldroyd constitutive law (see \cite{Oldroyd}). About the derivation of the system \eqref{mmm}, the interested readers can refer to \cite{lin2012}, here we omit it.

The study of the incompressible Oldroyd-B model started by a pioneering paper by Guillop\'{e} and Saut in \cite{GS}, the authors proved that the strong solutions are local well-posed in the
Sobolev space $H^s$. They \cite{GS2} also showed that these solutions are global if the coupling
parameter and the initial data are small enough.
The extensions of these results to the $L^p$-setting were given by Fern\'andez-Cara, Guill\'en and Ortega.
 In 2001, Chemin and Masmoudi \cite{chemin}  studied the local and global well-posedness of (\ref{mmm}) in $\R^n$. But for the global result in $L^p$ framework, they needed the small coupling parameter. This gap was filled in a  recent work by Zi, Fang and Zhang \cite{zuiruizhao} with the method based on Green's matrix of the linearized system.  Furthermore, some blowup criteria for local solutions were also established in \cite{chemin}. An improvement of the Chemin and Masmoudi's blow-up criterion was presented  by Lei, Masmoudi and Zhou \cite{leizhenjde}. In addition, we can refer to \cite{chenqionglei1} for more global existence results in generalized spaces.
 Very recently, for the case $\nu=0$,  adding an exact diffusion $-\mu \Delta \tau$ in the second equation of  system \eqref{mmm}, Elgindi and  Rousset \cite{elgindicpam} proved the global existence of smooth solutions in $\R^2$  with  the arbitrary initial data for $Q(\tau,\nabla u)=0$, $K_2\in\R$, $K_1,\beta\ge0$ and small initial data for $Q(\tau,\nabla u)\neq0$,  $K_1>0, K_2>0, \beta>0$, respectively. For the same model considered in \cite{elgindicpam},   the global small solutions in $\R^3$ was further obtained
 by Elgindi and Liu \cite{elgindijde} in Sobolev spaces
 $H^s(\R^3), s>\frac{5}{2}$.
Most recently, by
constructing two special time-weighted energies,
Zhu  \cite{zhuyi} obtained the global smooth  solutions to  system \eqref{mmm} with $ \beta=0$, $\mu>0, K_1>0, K_2>0$ in $\R^3$.
For  the same model considered in \cite{zhuyi}, Chen and Hao \cite{chenqionglei} established the global small solutions in critical Besov spaces for more general dimension $\R^n (n\ge 2)$. This result was further extended by Zhai \cite{zhaixiaoping} to  a critical $L^p$ framework
   which implies that
 a class of highly oscillating
initial data are allowed.
 There are other lots of excellent works have been done to Oldroyd model and related models (see \cite{chenqionglei1},   \cite{constanin},
\cite{Larson},     \cite{LZ}--\cite{zhangping2015}, \cite{zhaixiaoping2}).

It  should mention that,  in  Remark 2 of the paper by Elgindi and Liu  \cite{elgindijde}, the authors expect the global small solutions of the generalized Oldroyd type models with $\Lambda^{\alpha_1}$ dissipation on $u$ and $\Lambda^{\alpha_2}$ dissipation on $\tau$ in the two and three dimensions,
which  motivates us to consider the following
 generalized  Oldroyd-B model without damping mechanism:
\begin{eqnarray}\label{m}
\left\{\begin{aligned}
&\partial_t u+ u\cdot\nabla u+\nu \Lambda^\alpha u+\nabla\pi- K_1 \div \tau=0,\\
&\partial_t\tau + u\cdot \nabla \tau   + Q(\tau, \nabla u)  - K_2 D (u)=0,\\
&\div u =0,
\end{aligned}\right.
\end{eqnarray}
with $\Lambda\stackrel{\mathrm{def}}{=}\sqrt{-\Delta},$
 and  initial data satisfy
$$u(x,0)=u_0(x),\quad \tau(x,0)=\tau_0(x).$$


Now, we can state the  main theorem of the paper:

\begin{theorem}\label{th1}
Let $n\ge 2,$ $\nu,  K_1, K_2 >0$.  For any $1<\alpha\le 2$,  $\tau_0\in \dot{B}_{2,1}^{\frac n2+1-\alpha}\cap\dot{B}_{2,1}^{\frac {n}{2}}(\R^n)$, $u_0\in \dot{B}_{2,1}^{\frac n2+1-\alpha}(\R^n)$ with $\div u_0 = 0$.
 If  there exists a positive constant $c_0$ such that,
\begin{align}\label{smallness2}
\|u_0\|_{\dot{B}_{2,1}^{\frac n2+1-\alpha}}+\|\tau_0\|_{\dot{B}_{2,1}^{\frac n2+1-\alpha}\cap\dot{B}_{2,1}^{\frac {n}{2}}}\leq c_0,
\end{align}
then
the system \eqref{m} has a unique global solution $(u,\tau)$ so that for any $T>0$
\begin{align*}
&u\in C_b([0,T );{\dot{B}}_{2,1}^{\frac {n}{2}+1-\alpha}(\R^n))\cap L^{1}
([0,T];{\dot{B}}_{2,1}^{1+\frac {n}{2}}(\R^n)),\
\tau\in C_b([0,T );\dot{B}_{2,1}^{\frac n2+1-\alpha}\cap\dot{B}_{2,1}^{\frac {n}{2}}(\R^n)),\\
&(\Lambda^{-1}\p\div\tau)^\ell\in L^{1}
([0,T];{\dot{B}}_{2,1}^{\frac n2+1}(\R^n)),\quad(\Lambda^{-1}\p\div\tau)^h\in L^{1}
([0,T];{\dot{B}}_{2,1}^{\frac n2+2-\alpha}(\R^n)).
\end{align*}
Here $\p=\mathcal{I}-\mathcal{Q}\stackrel{\mathrm{def}}{=}\mathcal{I}-\nabla\Delta^{-1}\div$  is the projection operator. One refers to Section 2 for the definitions of   the Besov space $\dot{B}_{2,1}^{s}(\R^n)$ and $f^\ell$, $f^h.$
\end{theorem}

\begin{remark}\label{13323}
In Remark 1.3 of the paper by  Elgindi and Rousset in \cite{elgindicpam}, they expect the global solutions of the generalized version of \eqref{m} in $\R^2$  with $\Lambda^{2\alpha_1}u$, $\Lambda^{2\alpha_2}\tau$ and   $\alpha_1+\alpha_2=1,$ if neglecting the effect of the quadratic form $Q(\tau, \nabla u)$. To our knowledge, this problem is still open. The Theorem \ref{th1} brings us closer to this more interesting case, we hope we can give a positive answer about this problem in the further.
\end{remark}
\begin{remark}\label{23}
If $\alpha=2$, the above Theorem \ref{th1} coincides with the results by Chen and Hao  \cite{chenqionglei} and Zhai  \cite{zhaixiaoping}.
 Compared to \cite{chenqionglei}, \cite{zhaixiaoping}, we obtain the global small solutions with less dissipation for $u$.
\end{remark}

\begin{remark}\label{26}
For any $0<\alpha\le 1$,  we cannot obtain any smoothing effect of $\Lambda^{-1}\p\div\tau$, thus, we don't close our energy   estimates, a new method need to be introduced to deal with this difficulty.
\end{remark}
\begin{remark}\label{2222}
Let us finally say a few words on our functional setting.
As we all know that an important quantity for the wellposedness  problem of the fluid is $\|\nabla u\|_{L^1_T(L^\infty)}$.  By using the embedding relation ${\dot{B}}_{2,1}^{\frac {n}{2}}(R^n)\hookrightarrow L^\infty(\R^n)$, we only need to control the norm of $\|\nabla u\|_{L^1_T({\dot{B}}_{2,1}^{\frac {n}{2}})}.$ Now from the maximal smoothing effect of the fractional heat equation in \eqref{m}, the best space of the initial velocity should be ${\dot{B}}_{2,1}^{\frac {n}{2}+1-\alpha}(R^n)$.

\end{remark}

Next, we shall prove a global result of inviscid limit of the following model in the Sobolev spaces:
\begin{eqnarray}\label{m1}
\left\{\begin{aligned}
&\partial_t \u+ \u\cdot\nabla \u-\nu \Delta \u+\nabla \pi_\nu-  \div \tau_\nu=0,\\
&\partial_t\tau_\nu + \u\cdot \nabla \tau_\nu   -\Delta\tau_\nu-  D (\u)=0,\\
&\div \u =0,\quad\quad\quad\quad\quad\quad\quad\quad\quad\quad\quad\quad\quad\quad x\in \R^2, \quad t>0,\\
&(u_\nu,\tau_\nu)|_{t=0}=(u_0,\tau_0).
\end{aligned}\right.
\end{eqnarray}
Let $\nu=0$ in \eqref{m1}, that is, considering the following system:
\begin{eqnarray}\label{m2}
\left\{\begin{aligned}
&\partial_t u+ u\cdot\nabla u+\nabla \pi-  \div \tau=0,\\
&\partial_t\tau + u\cdot \nabla \tau   -\Delta\tau-  D (u)=0,\\
&\div u =0,\\
&(u,\tau)|_{t=0}=(u_0,\tau_0).
\end{aligned}\right.
\end{eqnarray}
 Elgindi and  Rousset \cite{elgindicpam} obtained the following theorem of \eqref{m2}:

\begin{theorem}\label{global}(see \cite{elgindicpam})
Assume that $(u_0,\tau_0)\in H^s(\R^2)$ with $s\in(2,\infty)$. Then  \eqref{m2} has a unique global solution $(u,\tau)$ satisfying , for any $T>0,$
$$(u,\tau)\in C([0,T];H^s(\R^2)).$$
\end{theorem}

The aim of this paper is to analyze the inviscid limit problem and we will show strong convergence of the solutions $(\u,\tau_\nu)$ of the system \eqref{m1} to the one of \eqref{m2} in the same space of initial data. More precisely, we obtain the following theorem:
\begin{theorem}\label{th2}
Let $\nu>0$ and $(u_0,\tau_0)\in H^\sigma(\R^2)$ with $\sigma>4.$  For any $T>0,$ there exists $\nu_0=\nu_0(T)>0$ such that, for $0<\nu\le\nu_0,$ \eqref{m1} has  a unique global solution satisfying $(\u,\tau_\nu)\in C([0,T];H^\sigma(\R^2)).$
Moreover, for any $0\le s \le \sigma-2$ and $0<\nu\le \nu_0,$ we have
$$\lim_{\nu\to0}\|(\u,\tau_\nu)-(u,\tau)\|_{L^\infty_T(H^s)}=0.$$
\end{theorem}

\begin{remark}\label{shoulian}
From the proof of the above theorem, we can get the precise convergence rate  in the $H^s$ space. More precisely, we have
 $$\|(\u,\tau_\nu)-(u,\tau)\|_{H^s}\le C(T)\nu,$$
where $ C(T)$ is  a constant dependent on $T$ and $\|(u,\tau)\|_{L^\infty([0,T];H^{\sigma})}$.

\end{remark}


\noindent$\mathbf{Notations:}$ Let $A$, $B$ be two operators, we denote $[A, B] = AB - BA$, the commutator
between $A$ and $B$. For $a\lesssim b$, we mean that there is a uniform constant $C$, which may
be different on different lines, such that $a \le C b$. We shall denote by$\langle a,b\rangle$ the $L^2(\R^d)$ inner product of $a$ and $b$.
For $X$ a Banach space and $I$ an interval of $\mathbb{R}$, we denote by $C(I; X)$ the set of
continuous functions on $I$ with values in $X$. For $q \in [1, +\infty]$, the notation $L^q (I; X)$ stands for the set of measurable
functions on $I$ with values in $X$, such that $t \rightarrow \|f(t)\|_{ X }$ belongs to $L^q (I)$.
We always let $(d_j)_{j\in\mathbb{Z}}$ be a
generic element of $\ell^1(\mathbb{Z})$  so that $\sum_{j\in\mathbb{Z}}d_j=1$.

\section{ Preliminaries}
The result of the present paper rely on the use of  a
dyadic partition of unity with respect to the Fourier variables, the so-called the
\textit{Littlewood-Paley  decomposition}. Here,  we  only
 give an homogeneous Littlewood-Paley decomposition
$(\ddj)_{j\in\Z}$ that is a dyadic decomposition in the Fourier space for $\R^n.$
One may for instance set $\ddj\stackrel{\mathrm{def}}{=}\varphi(2^{-j}D)$, $\dot{S}_j\stackrel{\mathrm{def}}{=}\chi(2^{-j}D)$ with $\varphi(\xi)\stackrel{\mathrm{def}}{=}\chi(\f\xi2)-\chi(\xi),$
and $\chi$ a non-increasing nonnegative smooth function supported in $B(0,\f43),$ and with
value $1$ on $B(0,\f34)$  (see \cite{bcd}, Chap. 2 for more details).

The Littlewood-Paley decomposition is foundational to the definition of Besov spaces.
\begin{definition}\label{d:besov}
 For $s\in\R$, $1\le p\le\infty,$   the homogeneous Besov space $\dot B^s_{p,1}(\R^n)$ is  the
set of tempered distributions $f$ satisfying
$$\lim_{j\rightarrow-\infty}\|\dot S_jf\|_{L^\infty}=0,\quad \hbox{and}\quad
\|f\|_{\dot B^s_{p,1}}\triangleq\sum_{j\in\Z} 2^{js}
\|\dot{\Delta}_j  f\|_{L^p}<\infty.$$
\end{definition}

For any $f \in\mathcal{S}'(\R^n)$, the lower and higher oscillation parts can be expressed as
\begin{align}\label{gaodi}
f^\ell\stackrel{\mathrm{def}}{=}\sum_{j\leq N_0}\ddj f\quad\hbox{and}\quad
f^h\stackrel{\mathrm{def}}{=}\sum_{j>N_0}\ddj f
\end{align}
for a large  integer $N_0\ge 0.$  
The corresponding  truncated semi-norms are defined  as follows:
\begin{align*}\|f\|^{\ell}_{\dot B^{s}_{p,1}}\stackrel{\mathrm{def}}{=}  \|f^{\ell}\|_{\dot B^{s}_{p,1}}
\ \hbox{ and }\   \|f\|^{h}_{\dot B^{s}_{p,1}}\stackrel{\mathrm{def}}{=}  \|f^{h}\|_{\dot B^{s}_{p,1}}.
\end{align*}

As we shall work with time-dependent functions valued in Besov spaces,
we  introduce the norms:
$$
\|f\|_{L^q_T(\dot B^s_{p,1})}\stackrel{\mathrm{def}}{=}\bigl\| \|f(t,\cdot)\|_{\dot B^s_{p,1}}\bigr\|_{L^q(0,T)}.
$$
Moreover,
in the present paper, we frequently use the so-called ''time-space" Besov spaces
or Chemin-Lerner space first introduced by Chemin and Lerner \cite{bcd}.
\begin{definition}
Let $s\in\mathbb{R}$ and $0<T\leq +\infty$. We define
$$
\|f\|_{\widetilde{L}_{T}^{q}(\dot{B}_{p,1}^{s})}\stackrel{\mathrm{def}}{=}
\sum_{j\in\mathbb{Z}}
2^{js}\left(\int_{0}^{T}\|\dot{\Delta}_{j}f(t)\|_{L^p}^{q}dt\right)^{\frac{1}{q}}
$$
for  $ q,\,p \in [1,\infty)$ and with the standard modification for $p,\,q =\infty $.
\end{definition}
\begin{remark}
For any $1\le q\le \infty,$
from the Minkowski's inequality, on can deduce that
$$\|f\|_{L^q_{T}(\dot{B}_{p,1}^s)}\le\|f\|_{\widetilde{L}^q_{T}(\dot{B}_{p,1}^s)}.$$
\end{remark}

The following lemma describes the way derivatives act on spectrally localized functions.
\begin{lemma}\label{bernstein}
Let $\mathcal{B}$ be a ball and $\mathcal{C}$ a ring of $\mathbb{R}^n$. A constant $C$ exists so that for any positive real number $\lambda$, any
non-negative integer k, any smooth homogeneous function $\sigma$ of degree m, and any couple of real numbers $(p, q)$ with
$1\le p \le q$, there hold
\begin{align*}
&&\mathrm{Supp} \,\hat{u}\subset\lambda \mathcal{B}\Rightarrow\sup_{|\alpha|=k}\|\partial^{\alpha}u\|_{L^q}\le C^{k+1}\lambda^{k+n(\f 1p-\f 1q)}\|u\|_{L^p},\\
&&\mathrm{Supp} \,\hat{u}\subset\lambda \mathcal{C}\Rightarrow C^{-k-1}\lambda^k\|u\|_{L^p}\le\sup_{|\alpha|=k}\|\partial^{\alpha}u\|_{L^p}
\le C^{k+1}\lambda^{k}\|u\|_{L^p},\\
&&\mathrm{Supp} \,\hat{u}\subset\lambda \mathcal{C}\Rightarrow \|\sigma(D)u\|_{L^q}\le C_{\sigma,m}\lambda^{m+n(\f 1p-\f 1q)}\|u\|_{L^p}.
\end{align*}
\end{lemma}

Next we  recall a few nonlinear estimates in Besov spaces which may be
obtained by means of paradifferential calculus.
Here, we recall the decomposition in the homogeneous context:
\begin{align}\label{bony}
uv=\dot{T}_uv+\dot{T}_vu+\dot{R}(u,v),
\end{align}
\noindent$\mathrm{where}\quad \dot{T}_uv\triangleq\sum_{j\in Z}\dot{S}_{j-1}u\dot{\Delta}_jv, \hspace{0.5cm}\dot{R}(u,v)\triangleq\sum_{j\in Z}
\dot{\Delta}_ju\widetilde{\dot{\Delta}}_jv,\quad
\hbox{and}\quad
 \widetilde{\dot{\Delta}}_jv\triangleq\sum_{|j-j'|\le1}\dot{\Delta}_{j'}v .$

The paraproduct $\dot{T}$ and the remainder $\dot{R}$ operators satisfy the following
continuous properties.

\begin{lemma}[\cite{bcd}]\label{fangji}
For all $s\in\mathbb{R}$, $\sigma\ge0$, and $1\leq p, p_1, p_2\leq\infty,$ the
paraproduct $\dot T$ is a bilinear, continuous operator from $\dot{B}_{p_1,1}^{-\sigma}\times \dot{B}_{p_2,1}^s$ to
$\dot{B}_{p,1}^{s-\sigma}$ with $\frac{1}{p}=\frac{1}{p_1}+\frac{1}{p_2}$. The remainder $\dot R$ is bilinear continuous from
$\dot{B}_{p_1, 1}^{s_1}\times \dot{B}_{p_2,1}^{s_2}$ to $
\dot{B}_{p,1}^{s_1+s_2}$ with
$s_1+s_2>0$, and $\frac{1}{p}=\frac{1}{p_1}+\frac{1}{p_2}$.
\end{lemma}

Next, we give the important product acts on homogenous Besov spaces and composition estimate which will be also often used implicitly throughout the paper.
\begin{lemma}\label{daishu}
Let  $s_1\leq \frac{n}{2}$, $s_2\leq \frac n2$ and $s_1+s_2>0$. For any $u\in\dot{B}_{2,1}^{s_1}({\mathbb R} ^n)$, $v\in\dot{B}_{2,1}^{s_2}({\mathbb R} ^n)$, we have
\begin{align*}
\|uv\|_{\dot{B}_{2,1}^{s_1+s_2 -\frac{n}{2}}}\lesssim \|u\|_{\dot{B}_{2,1}^{s_1}}\|v\|_{\dot{B}_{2,1}^{s_2}}.
\end{align*}
\end{lemma}


\begin{lemma}(Lemma 2.100 in \cite{bcd})\label{jiaohuanzi}
Let  $-1-\frac n2<s\leq 1+\frac n2$,
$v\in \dot{B}_{2,1}^{s}(\R^n)$ and $u\in \dot{B}_{2,1}^{\frac{n}{2}+1}(\R^n)$ with $\div u=0$.
Then there holds
$$
\big\|[\dot{\Delta}_j,u\cdot \nabla ]v\big\|_{L^2}\lesssim d_j 2^{-js}\|\nabla u\|_{\dot{B}_{2,1}^{\frac{n}{2}}}\|v\|_{\dot{B}_{2,1}^{s}}.
$$
\end{lemma}

\begin{remark}\label{com}
Let $A(D)$ be a zero-order Fourier multiplier.
For any $0\le \alpha<n+2,$ from the above lemma, we can easily get
\begin{align*}
\big\|[\dot{\Delta}_jA(D), u\cdot \nabla ]v\big\|_{L^2}\lesssim d_j 2^{-({\frac{n}{2}+1-\alpha})j}\|\nabla u\|_{\dot{B}_{2,1}^{\frac{n}{2}}}\|v\|_{\dot{B}_{2,1}^{\frac{n}{2}+1-\alpha}}.
\end{align*}
\end{remark}

Finally, we present a new commutator estimate in the Besov spaces which can be regarded as a partial generalization of the commutator estimate of Kato and Ponce \cite{ponce}.
In fact, it was proved in Kato and Ponce \cite{ponce},  for $s\ge0 $ and $1<p<\infty,$ that,
\begin{align}\label{kato}
\big\| J^{s} (u\cdot\nabla{v})-u\cdot\nabla J^{s}{v}\big\|_{L^p}\le C(\|\nabla u\|_{L^\infty}\|J^{s-1}\nabla v\|_{L^p}+\|\nabla v\|_{L^\infty}\|J^{s}u\|_{L^p}),
\end{align}
with $J^{s}f\stackrel{\mathrm{def}}{=}\mathcal{F}^{-1}[(1+|\xi|^2)^{\frac{s}{2}}\hat{f}(\xi))]$.

Especially,  let $p=2$ and $s>\frac n2$ in \eqref{kato}, one can get
\begin{align}\label{kato2}
\big\| J^{s} (u\cdot\nabla{v})-u\cdot\nabla J^{s}{v}\big\|_{L^2}\le C(\|\nabla u\|_{H^s}\| v\|_{H^s}+\|u\|_{H^s}\|\nabla v\|_{H^s}).
\end{align}
The estimate \eqref{kato2} was further improved by Fefferman {\it et al.} in \cite{Fefferman} to
\begin{align}\label{kato3}
\big\| \Lambda^{s} (u\cdot\nabla{v})-u\cdot\nabla\Lambda^{s}{v}\big\|_{L^2}\le C\|\nabla u\|_{H^s}\| v\|_{H^s},\quad s>\frac n2.
\end{align}
Moreover, they exhibited a counterexample in Appendix A in \cite{Fefferman} to show that the commutator estimate \eqref{kato3} does not hold in the case $s=\frac n2$, at least for $n=2$, even if $u$ and $v$ are required to be divergence-free.
In the Sobolev spaces, the estimate \eqref{kato3} seems optimal as the relation $H^{\frac n2}(\R^n)\hookrightarrow L^{\infty}(\R^n)$ is not valid.
However, if we consider the commutator estimate in the Besov spaces, we can verify that the estimate is still valid for  any dimension $n\ge 2$ if $s=\frac n2$ and $\div u = 0$.
More precisely, we shall prove  the following commutator estimate:
\begin{lemma}\label{xin}
Let $ n\geq2$,  $\Lambda\stackrel{\mathrm{def}}{=}\sqrt{-\Delta}$ and $-1\le s<n+1$.  For any   ${v}\in\dot{B}^{\frac{n}{2}}_{2,1}(\R^n)$, $\nabla u\in \dot{B}^{\f n2}_{2,1}(\R^n)$ with $\div u = 0$,  there exists a constant $C$ such that
\begin{align}\label{xin2}
\big\| \Lambda^{s} (u\cdot\nabla{v})-u\cdot\nabla\Lambda^{s}{v}\big\|_{\dot{B}^{\f n2-s}_{2,1}}\le C \big\|\nabla u\big\|_{\dot{B}^{\f n2}_{2,1}}\big\|{v}\big\|_{\dot{B}^{\f n2}_{2,1}}.
\end{align}
\end{lemma}
\begin{proof}
Firstly, according to the definition of the commutator, one can write
\begin{align}\label{}
\ddj([ \Lambda^{s}, u\cdot\nabla]{v})=[\ddj \Lambda^{s}, u\cdot\nabla]{v}-[\ddj , u\cdot\nabla]\Lambda^{s}{v}.
\end{align}
Applying Lemma \ref {jiaohuanzi},
  we can get the estimate of the  second term of the right-hand side of the above equation:
\begin{align}\label{yuyi10}
\big\|[\dot{\Delta}_j, u\cdot \nabla ]\Lambda^{s}{v}\big\|_{L^2}\lesssim d_j 2^{-({\frac{n}{2}-s})j}\|\nabla u\|_{\dot{B}_{2,1}^{\frac{n}{2}}}\|{v}\|_{\dot{B}_{2,1}^{\frac{n}{2}}},\quad -1\le s<n+1.
\end{align}

In the following, we main concern with the term about $[\ddj \Lambda^{s}, u\cdot\nabla]{v}$. With the aid of
the notion of para-products \eqref{bony}, we can easily write
\begin{align*}
[\ddj \Lambda^{s}, u\cdot\nabla]{v}=I_j^1+I_j^2+I_j^3,
\end{align*}%
where
\begin{align*}
&I_j^1=\sum_{|k-j|\leq4}[\dot{\Delta}_j\Lambda^s, \dot{S}_{k-1}u\cdot\nabla ] \dot{\Delta}_kv,\\
&I_j^2=\sum_{|k-j|\leq4}\dot{\Delta}_j\Lambda^s (\dot{\Delta}_ku\cdot\nabla \dot{S}_{k-1}v)-\sum_{k\geq j+1} \dot{\Delta}_ku \cdot\nabla \dot{S}_{k-1}\dot{\Delta}_j\Lambda^sv,\\
&I_j^3=\sum_{k\geq j-3}\dot{\Delta}_j\Lambda^s (\dot{\Delta}_ku\cdot\nabla \widetilde{\dot{\Delta}}_kv)-\sum_{|k-j|\leq2} \dot{\Delta}_ku\cdot\nabla \widetilde{\dot{\Delta}}_k\dot{\Delta}_j\Lambda^sv.
\end{align*}
with
$\widetilde{\dot{\Delta}_k}=\dot{\Delta}_{k-1}+\dot{\Delta}_k+\dot{\Delta}_{k+1}.$

Since
$$\widehat{\dot{\Delta}_j\Lambda^{s} f}=\varphi(2^{-j}\xi){|\xi|^{s}}\widehat{f},$$
one has
$$\dot{\Delta}_j\Lambda^{s} f=2^{j(n+s)}\mathfrak{h}(2^j\cdot)\star f,\ \ \rm {for \ some} \ \mathfrak{h}\in \mathcal{S}.$$
From the definition of Bony's decomposition and the mean value theorem, one can write $I_j^1$ into
\begin{align*}
I_j^1=&\sum_{|k-j|\le4}[\dot{\Delta}_j\Lambda^{s},\dot{S}_{k-1}{{u}_m} ]\partial_m\dot{\Delta}_{k}v\nonumber\\
=&2^{j(n+s)}\sum_{|k-j|\leq4}\int_{\R^n}h(2^jy)
(\dot{S}_{k-1}{u}_m(x-y)-\dot{S}_{k-1}{u}_m(x))\partial_m\dot{\Delta}_{k}v(x-y)dy
\nonumber\\=&-2^{j(n+s)}\sum_{|k-j|\leq4}\int_{\R^n}h(2^jy)
\Big(\int_0^1y\cdot\nabla \dot{S}_{k-1}{u}_m(x-\tau y)\,d\tau\Big) \partial_m\dot{\Delta}_{k}v(x-y)dy,
\end{align*}
from which we  can deduce that
\begin{align}\label{cc}
\sum_{j\in\mathbb{Z}}2^{j(\frac{n}{2}-s)}\|I_j^1\|_{L^2}\le& C\sum_{j\in\mathbb{Z}}2^{j(\frac{n}{2}-1)}\sum_{|k-j|\le4}\|\nabla \dot{S}_{k-1}u\|_{L^\infty}\|\nabla \dot{\Delta}_k{v}\|_{L^2}\nonumber\\
\le& C\|\nabla u\|_{L^\infty}\sum_{j\in\mathbb{Z}}2^{\frac{n}{2}j}\sum_{|k-j|\le4}\|\dot{\Delta}_k {v}\|_{L^2}
\le C\|\nabla u\|_{L^\infty}\|{v}\|_{\dot{B}^{\frac{n}{2}}_{2,1}}.
\end{align}
For the last two terms, we get by H\"{o}lder's inequality and Young's inequality that
\begin{align}\label{aa}
\sum_{j\in\mathbb{Z}}2^{j(\frac{n}{2}-s)}\|I_j^2\|_{L^2}\leq& C \sum_{j\in\mathbb{Z}}2^{j(\frac{n}{2}-s)}\sum_{|k-j|\leq4}2^{j(s+1)} \|\dot{S}_{k-1}v\|_{L^\infty}\|\dot{\Delta}_ku\|_{L^2}\nonumber\\
&\ \ +C\sum_{j\in\mathbb{Z}}2^{j(\frac{n}{2}-s)}\sum_{k\geq j+1}2^{j(s+1)} \|\dot{\Delta}_ku\|_{L^2}\|\dot{\Delta}_jv\|_{L^\infty}\nonumber\\
\leq& C\|v\|_{\dot{B}^{\frac{n}{2}}_{2,1}}\sum_{j\in\mathbb{Z}}2^{j(\frac{n}{2}+1)}
(\sum_{|k-j|\leq4}\|\dot{\Delta}_ku\|_{L^2}+\sum_{k\geq j+1} \|\dot{\Delta}_ku\|_{L^2})\nonumber\\
\leq &C\|u\|_{\dot{B}^{\frac{n}{2}+1}_{2,1}}\|v\|_{\dot{B}^{\frac{n}{2}}_{2,1}}.
\end{align}
Similarly,
\begin{align}\label{bb}
\sum_{j\in\mathbb{Z}}2^{j(\frac{n}{2}-s)}\|I_j^3\|_{L^2}\leq& C \sum_{j\in\mathbb{Z}}2^{j(\frac{n}{2}-s)}\sum_{k\geq j-3}2^{j(s+1)} \|\dot{\Delta}_ku\|_{L^2}\|\widetilde{\dot{\Delta}}_kv\|_{L^\infty}\nonumber \\
&\ \ +C\sum_{j\in\mathbb{Z}}2^{j(\frac{n}{2}-s)}\sum_{|k-j|\leq2}2^{j(s+1)} \|\dot{\Delta}_ku\|_{L^2}\|\dot{\Delta}_jv\|_{L^\infty}\nonumber\\
\leq & C\|v\|_{\dot{B}^{\frac{n}{2}}_{2,1}}\sum_{j\in\mathbb{Z}}2^{j(\frac{n}{2}+1)}
(\sum_{k\geq j-3}\|\dot{\Delta}_ku\|_{L^2}+\sum_{|k-j|\leq2} \|\dot{\Delta}_ku\|_{L^2})\nonumber\\
\leq &C\|u\|_{\dot{B}^{\frac{n}{2}+1}_{2,1}}\|v\|_{\dot{B}^{\frac{n}{2}}_{2,1}}.
\end{align}
Adding up \eqref{cc}, \eqref{aa}, and \eqref{bb}, we arrive at
\begin{align}\label{yuyi20}
\sum_{j\in\mathbb{Z}}2^{j(\frac{n}{2}-s)}[\ddj \Lambda^{s}, u\cdot\nabla]{v}\lesssim \|\nabla u\|_{\dot{B}_{2,1}^{\frac{n}{2}}}\|{v}\|_{\dot{B}_{2,1}^{\frac{n}{2}}}.
\end{align}
Combining \eqref{yuyi10} and \eqref{yuyi20}, we can  complete the proof of this lemma.
\end{proof}


\section{ The proof of the Theorem \ref{th1}}
Now, we begin to prove the global small solutions of \eqref{m}.
We first denote the initial  energy
\begin{align*}
\mathcal E_0(0)\stackrel{\mathrm{def}}{=}\|(u_0,\tau_0)\|_{\dot{B}_{2,1}^{\frac n2+1-\alpha}}+\|\tau_0\|_{\dot{B}_{2,1}^{\frac {n}{2}}\cap\dot{B}_{2,1}^{\frac n2+1-\alpha}},
\end{align*}
 and  then  define the total energy
\begin{align*}
\mathcal E(t)\stackrel{\mathrm{def}}{=}\mathcal E_1(t)+\mathcal E_2(t),
\end{align*}
\begin{align*}
\mathrm{with}\quad\quad \quad\quad \mathcal E_1(t)\stackrel{\mathrm{def}}{=}&\|(u,\tau)\|_{L^\infty_t(\dot{B}_{2,1}^{\frac n2+1-\alpha})}+\|\tau\|_{L^\infty_t(\dot{B}_{2,1}^{\frac {n}{2}})},\nonumber\\
\mathcal E_2(t)\stackrel{\mathrm{def}}{=}&\|u\|_{L^1_t(\dot{B}_{2,1}^{\frac {n}{2}+1})}
+\|(\Lambda^{-1}\p\div\tau)^\ell\|_{L^1_t(\dot{B}_{2,1}^{\frac {n}{2}+1})}+\|(\Lambda^{-1}\p\div\tau)^h\|_{L^1_t(\dot{B}_{2,1}^{\frac {n}{2}+2-\alpha})}.
\end{align*}
We shall derive the a priori estimates of $\mathcal E_1(t)$ and $\mathcal E_2(t)$ separatively.
\subsection{The estimates of $\mathcal E_1(t)$}

Applying   $\dot{\Delta}_j\p$ to the first equation and  $\dot{\Delta}_j$ to the second equation in  \eqref{m}, respectively, we discover that
\begin{eqnarray}\label{caihong1}
\left\{\begin{aligned}
&\partial_t \dot{\Delta}_ju+ u\cdot\nabla \dot{\Delta}_ju+ \dot{\Delta}_j\Lambda^\alpha u-\dot{\Delta}_j\p\div \tau=[u\cdot\nabla, \dot{\Delta}_j\p]u,\\
&\partial_t \dot{\Delta}_j\tau+ u\cdot \nabla \dot{\Delta}_j\tau   +\dot{\Delta}_j Q(\tau, \nabla u)  -  \dot{\Delta}_jD (u)=[u\cdot\nabla, \dot{\Delta}_j]\tau.
\end{aligned}\right.
\end{eqnarray}
Testing the first equation  and  second equation in  \eqref{caihong1} by  $\dot{\Delta}_ju$ and $\dot{\Delta}_j\tau$, respectively, yields
\begin{align}\label{caihong3}
\frac{1}{2}\frac{d}{dt}\|\dot{\Delta}_ju\|_{L^2}^2+\| \dot{\Delta}_j\Lambda^{\frac{\alpha}{2}}u\|_{L^2}^2=\langle\dot{\Delta}_j\p\div \tau, \dot{\Delta}_ju\rangle+\langle[u\cdot\nabla, \dot{\Delta}_j\p] u,\dot{\Delta}_j u\rangle,
\end{align}
and
\begin{align}\label{caihong5}
\frac{1}{2}\frac{d}{dt}\|\dot{\Delta}_j\tau\|_{L^2}^2=\langle \dot{\Delta}_jD (u), \dot{\Delta}_j\tau\rangle+\langle[u\cdot\nabla, \dot{\Delta}_j] \tau,\dot{\Delta}_j \tau\rangle
-\langle\dot{\Delta}_j Q(\tau, \nabla u),\dot{\Delta}_j\tau\rangle,
\end{align}
in which we have used the following fact:
 $$\langle u\cdot\nabla \dot{\Delta}_ju,\dot{\Delta}_ju\rangle=0,\quad
\langle u\cdot\nabla \dot{\Delta}_j\tau,\dot{\Delta}_j\tau\rangle=0.$$

Summing up  \eqref{caihong3} and \eqref{caihong5} and using the cancellation relation
$$\langle\dot{\Delta}_j\p\div \tau, \dot{\Delta}_ju\rangle+\langle \dot{\Delta}_jD (u), \dot{\Delta}_j\tau\rangle=0,$$
we  can get
\begin{align}\label{caihong4}
&\frac{1}{2}\frac{d}{dt}\|(\dot{\Delta}_ju,\dot{\Delta}_j\tau)\|_{L^2}^2+\| \dot{\Delta}_j\Lambda^{\frac{\alpha}{2}}u\|_{L^2}^2\nonumber\\
&\quad=\langle[u\cdot\nabla, \dot{\Delta}_j\p] u,\dot{\Delta}_j u\rangle+\langle[u\cdot\nabla, \dot{\Delta}_j] \tau,\dot{\Delta}_j \tau\rangle
+\langle\dot{\Delta}_j Q(\tau, \nabla u),\dot{\Delta}_j\tau\rangle,
\end{align}
which implies
\begin{align}\label{caihong6}
&\frac{1}{2}\frac{d}{dt}\|(\dot{\Delta}_ju,\dot{\Delta}_j\tau)\|_{L^2}^2
\lesssim \left|\langle[u\cdot\nabla, \dot{\Delta}_j\p] u,\dot{\Delta}_j u\rangle\right|+\left|\langle[u\cdot\nabla, \dot{\Delta}_j] \tau,\dot{\Delta}_j \tau\rangle\right|
+\left|\langle\dot{\Delta}_j Q(\tau, \nabla u),\dot{\Delta}_j\tau\rangle\right|.
\end{align}

Hence dividing  (formally) \eqref{caihong6} by $\|(\dot{\Delta}_ju,\dot{\Delta}_j\tau)\|_{L_2}$, multiplying by $2^{(\frac n2+1-\alpha)j}$, integrating \eqref{caihong6},   and summing over $j$, we obtain
\begin{align}\label{caihong7}
\|(u,\tau)\|_{\widetilde{L}_t^{\infty}(\dot{B}_{2,1}^{\frac n2+1-\alpha})}
\lesssim&\|(u_0,\tau_0)\|_{\dot{B}_{2,1}^{\frac n2+1-\alpha}}
+{\sum_{j\in\Z}2^{({\frac n2+1-\alpha})j}\|\dot{\Delta}_j Q(\tau, \nabla u)\|_{L^1_t(L^2)}}\nonumber\\
&+\sum_{j\in\Z}2^{({\frac n2+1-\alpha})j}(\|[u\cdot\nabla, \dot{\Delta}_j\p] u\|_{L^1_t(L^2)}+
\|[u\cdot\nabla, \dot{\Delta}_j] \tau\|_{L^1_t(L^2)}).
\end{align}
Applying Lemma \ref{daishu} gives
\begin{align}\label{caihong8}
{\sum_{j\in\Z}2^{({\frac n2+1-\alpha})j}\|\dot{\Delta}_j Q(\tau, \nabla u)\|_{L^1_t(L^2)}}\lesssim&\int^t_0\|\nabla u\|_{\dot{B}_{2,1}^{\frac {n}{2}}}\|\tau\|_{\dot{B}_{2,1}^{\frac n2+1-\alpha}}ds.
\end{align}
Next we see, thanks to Lemma \ref{jiaohuanzi} that
\begin{align}\label{nian6}
\sum_{j\in\Z}2^{({\frac n2+1-\alpha})j}(\|[u\cdot\nabla, \dot{\Delta}_j\p] u\|_{L^1_t(L^2)}+
\|[u\cdot\nabla, \dot{\Delta}_j] \tau\|_{L^1_t(L^2)})
\lesssim&\int^t_0\|\nabla u\|_{\dot{B}_{2,1}^{\frac {n}{2}}}\|\tau\|_{\dot{B}_{2,1}^{\frac n2+1-\alpha}}ds.
\end{align}

Taking estimates \eqref{caihong8} and \eqref{nian6} into \eqref{caihong7} gives
\begin{align}\label{nian8}
\|(u,\tau)\|_{\widetilde{L}_t^{\infty}(\dot{B}_{2,1}^{\frac n2+1-\alpha})}
\lesssim&\|(u_0,\tau_0)\|_{\dot{B}_{2,1}^{\frac n2+1-\alpha}}
+
\int^t_0\|\nabla u\|_{\dot{B}_{2,1}^{\frac {n}{2}}}\|(u,\tau)\|_{\dot{B}_{2,1}^{\frac n2+1-\alpha}}ds
.
\end{align}
From the first equation in \eqref{m}, we can get similarly to estimate \eqref{caihong7} that
\begin{align}\label{caihong9}
\|\tau\|_{\widetilde{L}_t^{\infty}(\dot{B}_{2,1}^{\frac n2})}
\lesssim&\|\tau_0\|_{\dot{B}_{2,1}^{\frac {n}{2}}}+{\sum_{j\in\Z}2^{\frac {jn}{2}}\| \dot{\Delta}_j D(u)\|_{L^1_t(L^2)}}\nonumber\\
&+{\sum_{j\in\Z}2^{\frac {jn}{2}}\|[u\cdot\nabla, \dot{\Delta}_j] \tau\|_{L^1_t(L^2)}}
+{\sum_{j\in\Z}2^{\frac {jn}{2}}\|\dot{\Delta}_j Q(\tau, \nabla u)\|_{L^1_t(L^2)}}.
\end{align}
Thanks to  Lemmas \ref{daishu}, \ref{jiaohuanzi}, we have
\begin{align}
&{\sum_{j\in\Z}2^{\frac {jn}{2}}\|\dot{\Delta}_j Q(\tau, \nabla u)\|_{L^1_t(L^2)}}+
&
\sum_{j\in\Z}2^{\frac {jn}{2}}\|\dot{\Delta}_j([u\cdot\nabla, \dot{\Delta}_j]\tau)\|_{L^1_t(L^2)}
\lesssim\int^t_0\|\nabla u\|_{\dot{B}_{2,1}^{\frac {n}{2}}}\|\tau\|_{\dot{B}_{2,1}^{\frac n2}}ds.\label{caihong11}
\end{align}
 Inserting the above estimate into \eqref{caihong9}, we can get
\begin{align}\label{caihong99}
\|\tau\|_{\widetilde{L}_t^{\infty}(\dot{B}_{2,1}^{\frac n2})}
\lesssim&\|\tau_0\|_{\dot{B}_{2,1}^{\frac {n}{2}}}+\int^t_0\| u\|_{\dot{B}_{2,1}^{1+\frac {n}{2}}}ds+
\int^t_0\|\nabla u\|_{\dot{B}_{2,1}^{\frac {n}{2}}}\|\tau\|_{\dot{B}_{2,1}^{\frac n2}}ds.
\end{align}

Together with \eqref{nian8} and \eqref{caihong99}, one has
\begin{align}\label{nian10}
&\|(u,\tau)\|_{\widetilde{L}_t^{\infty}(\dot{B}_{2,1}^{\frac n2+1-\alpha})}+\|\tau\|_{\widetilde{L}_t^{\infty}(\dot{B}_{2,1}^{\frac n2})}
\nonumber\\
&\quad\lesssim\|(u_0,\tau_0)\|_{\dot{B}_{2,1}^{\frac n2+1-\alpha}}+\|\tau_0\|_{\dot{B}_{2,1}^{\frac {n}{2}}}
\nonumber\\
&\quad\quad+\int^t_0\| u\|_{\dot{B}_{2,1}^{1+\frac {n}{2}}}ds
+
\int^t_0\|\nabla u\|_{\dot{B}_{2,1}^{\frac {n}{2}}}(\|(u,\tau)\|_{\dot{B}_{2,1}^{\frac n2+1-\alpha}}+\|\tau\|_{\dot{B}_{2,1}^{\frac n2}})ds
\end{align}
which implies
\begin{align}\label{energy1}
\mathcal E_1(t)\le C\mathcal E_0(0)+\mathcal E_2(t)+C\mathcal E_1(t)\mathcal E_2(t).
\end{align}

Due to lack of dissipation of $\tau$ in \eqref{m}, in the above basic energy-type argument, we have to abandon the smooth effect of $u$ tentatively
(i.e. a good term $\| \dot{\Delta}_j\Lambda^{\frac{\alpha}{2}}u\|_{L^2}^2$  from \eqref{caihong4} to  \eqref{caihong7}). We will get back the smooth effect of $u$ and the hidden dissipation of $\p\div\tau$ in the next subsection.

\subsection{The estimates of $\mathcal E_2(t)$}

As discussed in the above subsection, the aim of this subsection is to find the dissipation of $u$ and $\tau$. In fact, we can only get the dissipation of $u$  and $\p\div\tau$. A key observation is that the equations about $u$ and $\Lambda^{-1}\p\div\tau$ have a good structure which can make the dissipation transfer from $u$ to $\Lambda^{-1}\p\div\tau$. In order to express our idea more precisely, we first
apply project operator $\p$ on both sides of the first two equations in \eqref{m}  to get
\begin{eqnarray}\label{G1}
\left\{\begin{aligned}
&\partial_t  u +\p (u\cdot \nabla u)+\Lambda^{\alpha}   u-\p \div \tau=0,\\
&\partial_t \p \div \tau+\p \div(u\cdot \nabla \tau) -\Delta u+\p \div(Q(\tau, \nabla u))=0.
\end{aligned}\right.
\end{eqnarray}



In the following, we introduce two magic new quantities:
 $$\aa\stackrel{\mathrm{def}}{=}\Lambda^{-1}\p\div\tau,\quad  w\stackrel{\mathrm{def}}{=}\Lambda^{\alpha-1} \aa-u.$$
Let us note that
 the velocity $u$ has good smoothing effect, we can get information on $\phi$ from information on $w.$

By using the new definitions of $\aa$ and $w,$ we deduce  from \eqref{G1} that
\begin{eqnarray}\label{G5}
\left\{\begin{aligned}
&\partial_t \aa+ u\cdot \nabla\aa+\Lambda u=f ,\\
&\partial_t  u + u\cdot \nabla u+\Lambda^{\alpha}   u-\Lambda \aa=g,\\
&\partial_t w+ u\cdot \nabla w+\La\aa=F,
\end{aligned}\right.
\end{eqnarray}
in which
\begin{align*}
&f\stackrel{\mathrm{def}}{=}- [\Lambda^{-1}\p\div,u\cdot \nabla ]\tau-\Lambda^{-1}\p \div(Q(\tau, \nabla u)),\\
&g\stackrel{\mathrm{def}}{=}-[\p ,u\cdot \nabla] u,\quad\quad
F\stackrel{\mathrm{def}}{=}-[\La^{\alpha-1}, u\cdot \nabla]\aa+\La^{\alpha-1} f-g.
\end{align*}

Considering the linear system of \eqref{G5}, we find  $\phi, u$ satisfy some kind of damped wave equations and have enough decay. The result in \cite{zhuyi}
is just based on  this fact by the time-weighted energy method.  However, the method used in \cite{zhuyi} seems  not to be work in the lower regularity spaces here. We shall use the localization technique and different commutator argument to get desired a priori estimates.

We begin to
apply $\ddj $ to the first equation in \eqref{G5}  that
\begin{align}\label{G6}
\partial_t \ddj \aa+u\cdot \nabla\ddj\aa+\ddj\La u=-[\ddj,u\cdot \nabla]\aa+\ddj f.
\end{align}
Taking $L^2$ inner product of $\ddj \aa$ with   \eqref{G6} and using integrating by parts, we have
\begin{align}\label{G7}
&\f12\f{d}{dt}\|{\ddj \aa}\|_{L^2}^2+\int_{\R^n}{\ddj \La\aa} \cdot\ddj u \,dx=-\int_{\R^n} [\ddj,u\cdot \nabla]\aa\cdot{\ddj \aa}\,dx+\int_{\R^n} \ddj f \cdot{\ddj \aa}\,dx.
\end{align}
Similarly, from the second  and third equations in \eqref{G5}
\begin{align}
\f12\f{d}{dt}\|{\ddj u}\|_{L^2}^2+    \|{\ddj \La^{\frac\alpha2}u}\|_{L^2}^2-&\int_{\R^n} \ddj \La \aa\cdot \ddj u\,dx\nonumber\\
&=-\int_{\R^n} [\ddj,u\cdot \nabla]u\cdot{\ddj u}\,dx+\int_{\R^n} \ddj g \cdot{\ddj u}\,dx,\label{G8}\\
\f12\f{d}{dt}\|{\ddj w}\|_{L^2}^2+\|{\ddj \La^{\frac\alpha2}\aa}\|_{L^2}^2-&\int_{\R^n}\ddj \La \aa\cdot\ddj u \,dx\nonumber\\
&=-\int_{\R^n} [\ddj,u\cdot \nabla]w\cdot{\ddj w}\,dx+\int_{\R^n} \ddj F \cdot{\ddj w}\,dx,\label{G9}
\end{align}
in which we have used the fact:
$$
\int_{\R^n} \ddj \La \aa\cdot \ddj w\,dx
=\|{\ddj \La^{\frac\alpha2} \aa}\|_{L^2}^2-\int_{\R^n}\ddj \La \aa\cdot\ddj u \,dx.
$$

Let $0<\eta<1$  be a small constant which will be determined later on. Summing up \eqref{G7}--\eqref{G9} and using the H\"older inequality and Berntein's lemma, we have
\begin{align}\label{G12}
&\f12\f{d}{dt}(\|{\ddj \aa}\|_{L^2}^2+(1-\eta)\|{\ddj u}\|_{L^2}^2+\eta\|{\ddj w}\|_{L^2}^2)
+ (1-\eta)2^{\alpha j} \|{\ddj u}\|_{L^2}^2+\eta2^{\alpha j}\|{\ddj \aa}\|_{L^2}^2
\nonumber\\
&\quad\lesssim\|{\ddj \aa}\|_{L^2}(\|[\ddj,u\cdot \nabla]\aa\|_{L^2}+\|{\ddj f}\|_{L^2})\nonumber\\
&\quad\quad+
\|{\ddj u}\|_{L^2}(\|[\ddj,u\cdot \nabla]u\|_{L^2}+\|{\ddj g}\|_{L^2})+\|{\ddj w}\|_{L^2}(\|[\ddj,u\cdot \nabla]w\|_{L^2}+\|{\ddj F}\|_{L^2}).
\end{align}
By Lemma \ref{bernstein} and $w\stackrel{\mathrm{def}}{=}\Lambda^{\alpha-1} \aa-u$,  we can  get
\begin{align}\label{}
\|{\ddj w}\|_{L^2}\le\|\ddj(\Lambda^{\alpha-1} \aa)\|_{L^2}+\|\ddj u\|_{L^2}\le C2^{({\alpha-1})j}\|\ddj \aa\|_{L^2}+\|\ddj u\|_{L^2}.
\end{align}

For any $2^j\le N_0$, we can find an $\eta>0$ small enough such that
\begin{align}\label{G13}
\|{\ddj \aa}\|_{L^2}^2+(1-\eta)\|{\ddj u}\|_{L^2}^2+\eta\|{\ddj w}\|_{L^2}^2
\ge \frac{1}{C}(\|{\ddj \aa}\|_{L^2}^2+\|{\ddj u}\|_{L^2}^2).
\end{align}
From \eqref{G12} and \eqref{G13}, one can deduce that
\begin{align}\label{G14}
&\|({\ddj u},{\ddj \aa})\|_{L^2}
+ 2^{\alpha j}\int_0^t\|({\ddj  u},{\ddj \aa})\|_{L^2}ds
\nonumber\\
&\quad\lesssim\|(\ddj u_0,\ddj \aa_0)\|_{L^2}+\int_0^t\|({\ddj f},{\ddj g},{\ddj F})\|_{L^2}ds\nonumber\\
&\quad\quad+\int_0^t(
\|[\ddj,u\cdot \nabla]\aa\|_{L^2}+\|[\ddj,u\cdot \nabla]u\|_{L^2}+\|[\ddj,u\cdot \nabla]w\|_{L^2})ds.
\end{align}

Multiplying by $2^{(\frac n2+1-\alpha)j}$  and summing over $2^j\le N_0$, we can further get
\begin{align}\label{G15}
&\|(u^\ell,\aa^\ell)\|_{\widetilde{L}_t^{\infty}(\dot{B}_{2,1}^{1+\frac n2-\alpha})}
+\int^t_0\| (u^\ell,\aa^\ell)\|_{\dot{B}_{2,1}^{1+\frac {n}{2}}}ds\nonumber\\
&\quad\lesssim\|(u_0^\ell,\aa_0^\ell)\|_{\dot{B}_{2,1}^{1+\frac n2-\alpha}}
+
\int^t_0\|(f,F,g)^\ell\|_{\dot{B}_{2,1}^{1+\frac n2-\alpha}}ds\nonumber\\
&\quad\quad+\int^t_0\sum_{2^j\le N_0}2^{(1+\frac n2-\alpha)j}(\|[\ddj,u\cdot \nabla]\aa\|_{L^2}+\|[\ddj,u\cdot \nabla]u\|_{L^2}+\|[\ddj,u\cdot \nabla]\Lambda^{\alpha-1} \aa\|_{L^2})ds.
\end{align}

In the following, we main concern the estimates for the high frequency part of the solution. We get similarly to  \eqref{G8}, \eqref{G9}  that
\begin{align}
&\f12\f{d}{dt}\|{\ddj u}\|_{L^2}^2+  \|{\ddj\La^{\frac\alpha2} u}\|_{L^2}^2- \|{\ddj\La^{1-\frac\alpha2} u}\|_{L^2}^2\nonumber\\
&\quad=-\int_{\R^n} [\ddj,u\cdot \nabla]u\cdot{\ddj u}\,dx+\int_{\R^n} \ddj \La^{2-\alpha}w \cdot \ddj u\,dx+\int_{\R^n} \ddj g \cdot{\ddj u}\,dx,\label{G16}\\
&\f12\f{d}{dt}\|{\ddj w}\|_{L^2}^2+\|{\ddj\La^{1-\frac\alpha2} w}\|_{L^2}^2\nonumber\\
&\quad=-\int_{\R^n} [\ddj,u\cdot \nabla]w\cdot{\ddj w}\,dx-\int_{\R^n} \ddj \La^{2-\alpha}u\cdot \ddj w\,dx+\int_{\R^n} \ddj F \cdot{\ddj w}\,dx,\label{G17}
\end{align}
in which we have used:
\begin{align*}
\int_{\R^n}\ddj\Lambda\aa\cdot\ddj udx
=&\|{\ddj\La^{1-\frac\alpha2} u}\|_{L^2}^2+\int_{\R^n}\La^{2-\alpha}\ddj w\cdot\ddj udx,\nonumber\\
\int_{\R^n}\ddj\Lambda\aa\cdot\ddj wdx
=&\|{\ddj\La^{1-\frac\alpha2} w}\|_{L^2}^2+\int_{\R^n}\La^{2-\alpha}\ddj u\cdot\ddj wdx.
\end{align*}

Summing up \eqref{G16}, \eqref{G17} and choosing  a suitable $N_0$ (for example $N_0=2^{\frac{1}{2(\alpha-1)}})$ such that for any $2^{j}\ge N_0$ there holds
 $C(2^{\alpha j} -2^{(2-\alpha) j})\ge\frac{C}{2}2^{(2-\alpha )j},$  then we have
\begin{align}\label{G20}
&\f12\f{d}{dt}(\|{\ddj u}\|_{L^2}^2+\|{\ddj w}\|_{L^2}^2)+ 2^{(2-\alpha )j}( \|{\ddj u}\|_{L^2}^2+\|{\ddj w}\|_{L^2}^2)\nonumber\\
&\quad\lesssim\|{\ddj u}\|_{L^2}(\|{\ddj g}\|_{L^2}+\|[\ddj,u\cdot \nabla]u\|_{L^2})+(\|{\ddj F}\|_{L^2}+\|[\ddj,u\cdot \nabla]w\|_{L^2})\|{\ddj w}\|_{L^2}.
\end{align}

Multiplying by $2^{(\frac n2+1-\alpha)j}$  and summing over $2^j\ge N_0$ imply that
\begin{align}\label{G22}
&\|(u^h,w^h)\|_{\widetilde{L}_t^{\infty}(\dot{B}_{2,1}^{1+\frac n2-\alpha})}
+\int^t_0\| (u^h, w^h)\|_{\dot{B}_{2,1}^{\frac {n}{2}+3-2\alpha}}ds\nonumber\\
&\quad\lesssim\|(u_0^h,w_0^h)\|_{\dot{B}_{2,1}^{\frac {n}{2}+1-\alpha}}
+
\int^t_0\|(F,g)^h\|_{\dot{B}_{2,1}^{\frac n2+1-\alpha}}ds\nonumber\\
&\quad\quad+\int^t_0\sum_{2^j\ge N_0}2^{(\frac n2+1-\alpha)j}(\|[\ddj,u\cdot \nabla]u\|_{L^2}+\|[\ddj,u\cdot \nabla]w\|_{L^2})
ds
.
\end{align}
Due to $\aa=\La^{1-\alpha}w+\La^{1-\alpha}u$ and $1<\alpha\le 2$, we can get
\begin{align*}
&\|\aa^h\|_{\widetilde{L}_t^{\infty}(\dot{B}_{2,1}^{\frac n2})}=\|(\La^{1-\alpha}w+\La^{1-\alpha}u)^h\|_{\widetilde{L}_t^{\infty}(\dot{B}_{2,1}^{\frac n2})}
\lesssim\|(u^h,w^h)\|_{\widetilde{L}_t^{\infty}(\dot{B}_{2,1}^{1+\frac n2-\alpha})},\nonumber\\
&\|\aa^h\|_{{L}_t^{1}(\dot{B}_{2,1}^{\frac n2+2-\alpha})}=\|(\La^{1-\alpha}w+\La^{1-\alpha}u)^h\|_{{L}_t^{1}(\dot{B}_{2,1}^{\frac n2+2-\alpha})}
\lesssim\|(u^h,w^h)\|_{{L}_t^{1}(\dot{B}_{2,1}^{\frac n2+3-2\alpha})},
\end{align*}
from which and \eqref{G22}, we yield
\begin{align}\label{G22+1}
&\|u^h\|_{\widetilde{L}_t^{\infty}(\dot{B}_{2,1}^{1+\frac n2-\alpha})}+\|\aa^h\|_{\widetilde{L}_t^{\infty}(\dot{B}_{2,1}^{\frac n2})}
+\int^t_0(\| u^h\|_{\dot{B}_{2,1}^{\frac {n}{2}+3-2\alpha}}+\|  \aa^h\|_{\dot{B}_{2,1}^{\frac {n}{2}+2-\alpha}})ds\nonumber\\
&\quad\lesssim\|u_0^h\|_{\dot{B}_{2,1}^{\frac {n}{2}+1-\alpha}}+\|\aa_0^h\|_{\dot{B}_{2,1}^{\frac {n}{2}}}
+
\int^t_0\|(F,g)^h\|_{\dot{B}_{2,1}^{\frac n2+1-\alpha}}ds\nonumber\\
&\quad\quad+\int^t_0\sum_{2^j\ge N_0}2^{(\frac n2+1-\alpha)j}(\|[\ddj,u\cdot \nabla]u\|_{L^2}+\|[\ddj,u\cdot \nabla]w\|_{L^2})
ds
.
\end{align}

From the  equation
$\partial_t  u + u\cdot \nabla u+\Lambda^{\alpha}   u-\Lambda \aa=g$ and the smoothing effect of the parabolic equation with fractional derivative, one can find that  we have lost $''2\alpha-2''$ order regularity for $u^h$ in  \eqref{G22+1}.
The reason why this case happened is that $u$ and $\aa$ have different order  smoothing effect in the high frequencies part.
Although we have lost $''2\alpha-2''$ order regularity for $u^h$ in  \eqref{G22+1}, we have obtained   the whole dissipation of $\aa^h$ luckily.
From the second equation in \eqref{G5}, we can get similarly to \eqref{G22} that
\begin{align}\label{G22+2}
&\|u^h\|_{\widetilde{L}_t^{\infty}(\dot{B}_{2,1}^{1+\frac n2-\alpha})}
+\int^t_0\| u^h\|_{\dot{B}_{2,1}^{\frac {n}{2}+1}}ds\nonumber\\
&\quad\lesssim\|u_0^h\|_{\dot{B}_{2,1}^{\frac {n}{2}+1-\alpha}}
+
\int^t_0\| \aa^h\|_{\dot{B}_{2,1}^{\frac n2+2-\alpha}}ds\nonumber\\
&\quad\quad+
\int^t_0\|g^h\|_{\dot{B}_{2,1}^{\frac n2+1-\alpha}}ds+\int^t_0\sum_{2^j\ge N_0}2^{(\frac n2+1-\alpha)j}\|[\ddj,u\cdot \nabla]u\|_{L^2}ds.
\end{align}

Multiplying by a suitable large constant on both sides of \eqref{G22+1} and then plusing \eqref{G22+2}, we can get for the high frequency of the solutions
\begin{align}\label{G22+3}
&\|u^h\|_{\widetilde{L}_t^{\infty}(\dot{B}_{2,1}^{1+\frac n2-\alpha})}+\|\aa^h\|_{\widetilde{L}_t^{\infty}(\dot{B}_{2,1}^{\frac n2})}
+\int^t_0(\| u^h\|_{\dot{B}_{2,1}^{\frac {n}{2}+1}}+\|  \aa^h\|_{\dot{B}_{2,1}^{\frac {n}{2}+2-\alpha}})ds\nonumber\\
&\quad\lesssim\|u_0^h\|_{\dot{B}_{2,1}^{\frac {n}{2}+1-\alpha}}+\|\aa_0^h\|_{\dot{B}_{2,1}^{\frac {n}{2}}}
+
\int^t_0\|(F,g)^h\|_{\dot{B}_{2,1}^{\frac n2+1-\alpha}}ds\nonumber\\
&\quad\quad+\int^t_0\sum_{2^j\ge N_0}2^{(\frac n2+1-\alpha)j}(\|[\ddj,u\cdot \nabla]u\|_{L^2}+\|[\ddj,u\cdot \nabla]w\|_{L^2})
ds
.
\end{align}

Combining with \eqref{G15} and \eqref{G22+3}, using the definition of $w\stackrel{\mathrm{def}}{=}\Lambda^{\alpha-1} \aa-u$, we get
\begin{align}\label{G24}
&\|(u,\aa^\ell)\|_{\widetilde{L}_t^{\infty}(\dot{B}_{2,1}^{\frac n2+1-\alpha})}+\|\aa^h\|_{\widetilde{L}_t^{\infty}(\dot{B}_{2,1}^{\frac n2})}+\int^t_0\| \aa^h\|_{\dot{B}_{2,1}^{\frac {n}{2}+2-\alpha}}ds
+\int^t_0\| (u,\aa^\ell)\|_{\dot{B}_{2,1}^{1+\frac {n}{2}}}ds\nonumber\\
&\quad\lesssim\|(u_0,\aa_0^\ell)\|_{\dot{B}_{2,1}^{\frac n2+1-\alpha}}+\|\aa_0^h\|_{\dot{B}_{2,1}^{\frac {n}{2}}}
\nonumber\\
&\quad\quad+
\int^t_0\|(f,g,F)\|_{\dot{B}_{2,1}^{\frac n2+1-\alpha}}ds+\int^t_0\sum_{j\in\Z}2^{({\frac n2+1-\alpha})j}\|[\ddj,u\cdot \nabla]\aa\|_{L^2}ds
\nonumber\\
&\quad\quad
+\int^t_0\sum_{j\in\Z}2^{({\frac n2+1-\alpha})j}(\|[\ddj,u\cdot \nabla]u\|_{L^2}+\|[\ddj,u\cdot \nabla]\Lambda^{\alpha-1} \aa\|_{L^2})ds.
\end{align}

Next, we give the estimates to the terms of the right-hand side of the above inequality.
By a simple computation, one has
\begin{align}\label{}
\dot{\Delta}_j([\p ,u\cdot \nabla] u)=[\dot{\Delta}_j\p,u\cdot \nabla] u-[\dot{\Delta}_j,u\cdot \nabla] u.
\end{align}
As the operator $\p$ is a zero-order Fourier multiplier, we can deduce from the Lemma \ref{jiaohuanzi} and the Remark \ref{com} that
\begin{align}\label{nei}
\|g\|_{\dot{B}_{2,1}^{\frac n2+1-\alpha}}
=&\sum_{j\in\Z}2^{(\frac n2+1-\alpha)j}\|\dot{\Delta}_j([\p ,u\cdot \nabla] u)\|_{L^2}\nonumber\\
\lesssim&\sum_{j\in\Z}2^{(\frac n2+1-\alpha)j}(\|[\dot{\Delta}_j\p,u\cdot \nabla] u\|_{L^2}+\|[\dot{\Delta}_j,u\cdot \nabla] u\|_{L^2})\nonumber\\
\lesssim&\|\nabla u\|_{\dot{B}_{2,1}^{\frac n2}}\|u\|_{\dot{B}_{2,1}^{\frac n2+1-\alpha}}.
\end{align}

Due  to the operator $\Lambda^{-1}\p\div$ is also a zero-order Fourier multiplier, we can get by a similar derivation of \eqref{nei} that
\begin{align}
\|[\Lambda^{-1}\p\div,u\cdot \nabla ]\tau\|_{\dot{B}_{2,1}^{\frac n2}}
\lesssim&\|\nabla u\|_{\dot{B}_{2,1}^{\frac n2}}\|\tau\|_{\dot{B}_{2,1}^{\frac n2}},\label{nei2}\\
\|[\Lambda^{-1}\p\div,u\cdot \nabla ]\tau\|_{\dot{B}_{2,1}^{\frac n2+1-\alpha}}
\lesssim&\|\nabla u\|_{\dot{B}_{2,1}^{\frac n2}}\|\tau\|_{\dot{B}_{2,1}^{\frac n2+1-\alpha}}\label{nei3}.
\end{align}

From  Lemma \ref{daishu}, we have
\begin{align}
\|\Lambda^{-1}\p \div(Q(\tau, \nabla u))\|_{\dot{B}_{2,1}^{\frac n2}}\lesssim&\|\nabla u\|_{\dot{B}_{2,1}^{\frac n2}}\|\tau\|_{\dot{B}_{2,1}^{\frac n2}}, \label{G26-1}\\
\|\Lambda^{-1}\p \div(Q(\tau, \nabla u))\|_{\dot{B}_{2,1}^{\frac n2+1-\alpha}}\lesssim&\|\nabla u\|_{\dot{B}_{2,1}^{\frac n2}}\|\tau\|_{\dot{B}_{2,1}^{\frac n2+1-\alpha}}.\label{G26-2}
\end{align}
The  combination of \eqref{nei2}--\eqref{G26-1} gives
\begin{align}\label{G26}
\|f\|_{\dot{B}_{2,1}^{\frac n2+1-\alpha}}\lesssim\|\nabla u\|_{\dot{B}_{2,1}^{\frac n2}}\|\tau\|_{\dot{B}_{2,1}^{\frac n2+1-\alpha}},\quad
\|\La^{\alpha-1} f\|_{\dot{B}_{2,1}^{\frac n2+1-\alpha}}\lesssim\| f\|_{\dot{B}_{2,1}^{\frac n2}}\lesssim\|\nabla u\|_{\dot{B}_{2,1}^{\frac n2}}\|\tau\|_{\dot{B}_{2,1}^{\frac n2}}.
\end{align}
By Lemma \ref{xin}, we have
\begin{align*}
&\|[\La^{\alpha-1}, u\cdot \nabla]\aa\|_{\dot{B}_{2,1}^{\frac n2+1-\alpha}}\lesssim \|\nabla u\|_{\dot{B}_{2,1}^{\frac n2}}\|\aa\|_{\dot{B}_{2,1}^{\frac n2}} \lesssim \|\nabla u\|_{\dot{B}_{2,1}^{\frac n2}}\|\tau\|_{\dot{B}_{2,1}^{\frac n2}} .
\end{align*}
from which and  \eqref{nei}, \eqref{G26},  we  can get
\begin{align}\label{G28}
&\|F\|_{\dot{B}_{2,1}^{\frac n2+1-\alpha}}\lesssim\|\nabla u\|_{\dot{B}_{2,1}^{\frac n2}}(\|u\|_{\dot{B}_{2,1}^{\frac n2+1-\alpha}}+\|\tau\|_{\dot{B}_{2,1}^{\frac n2}}).
\end{align}
By using Lemma \ref{jiaohuanzi}, we can get
\begin{align}\label{G281}
&\sum_{j\in\Z}2^{(\frac n2+1-\alpha)j}(\|[\ddj,u\cdot \nabla]\aa\|_{L^2}+\|[\ddj,u\cdot \nabla]u\|_{L^2}+\|[\ddj,u\cdot \nabla]\La^{\alpha-1}\aa\|_{L^2})\nonumber\\
&\quad\lesssim\|\nabla u\|_{\dot{B}_{2,1}^{\frac n2}}(\|\aa\|_{\dot{B}_{2,1}^{\frac n2+1-\alpha}}+\|u\|_{\dot{B}_{2,1}^{\frac n2+1-\alpha}}+\|\La^{\alpha-1}\aa\|_{\dot{B}_{2,1}^{\frac n2+1-\alpha}})\nonumber\\
&\quad\lesssim\|\nabla u\|_{\dot{B}_{2,1}^{\frac n2}}(\|(u,\tau)\|_{\dot{B}_{2,1}^{\frac n2+1-\alpha}}+\|\tau\|_{\dot{B}_{2,1}^{\frac n2}}).
\end{align}

Inserting \eqref{nei}, \eqref{G26}, \eqref{G28} and \eqref{G281} into \eqref{G24} gives
\begin{align}\label{G30}
&\|(u,\aa^\ell)\|_{\widetilde{L}_t^{\infty}(\dot{B}_{2,1}^{\frac n2+1-\alpha})}+\|\aa^h\|_{\widetilde{L}_t^{\infty}(\dot{B}_{2,1}^{\frac n2})}+\int^t_0 \| \aa^h\|_{\dot{B}_{2,1}^{\frac {n}{2}+2-\alpha}}ds
+\int^t_0\| (u,\aa^\ell)\|_{\dot{B}_{2,1}^{1+\frac {n}{2}}}ds\nonumber\\
&\quad\lesssim\|(u_0,\tau_0^\ell)\|_{\dot{B}_{2,1}^{\frac n2+1-\alpha}}+\|\tau_0^h\|_{\dot{B}_{2,1}^{\frac {n}{2}}}
+
\int^t_0\|\nabla u\|_{\dot{B}_{2,1}^{\frac n2}}(\|(u,\tau)\|_{\dot{B}_{2,1}^{\frac n2+1-\alpha}}+\|\tau\|_{\dot{B}_{2,1}^{\frac n2}})ds
\end{align}
which implies that
\begin{align}\label{energy2}
\mathcal E_2(t)\le C\mathcal E_0(0)+C\mathcal E_1(t)\mathcal E_2(t).
\end{align}

\subsection{Proof of the Theorem \ref{th1}}

Now, let us  give the proof of Theorem \ref{th1}. Multiplying by a suitable large constant on both sides of \eqref{energy2} and then plusing \eqref{energy1}, we can get
\begin{align}\label{energy3}
\mathcal E(t)\le C_1\mathcal E_0(0)+C_1\mathcal E_1(t)\mathcal E_2(t))\le C_2\mathcal E_0(0)+C_2\mathcal E^2(t)
\end{align}
for some positive constant $C_2$.

Under the setting of initial data in Theorem\ref{th1},  there exists a positive constant $C_3$ such that
$ \mathcal E (0) + C_2 \mathcal E_0(0)\leq C_3 \varepsilon$. Due to the local existence result which can be achieved similarly to \cite{chenqionglei}, there exists a positive time $T$ such that
\begin{equation}\label{re}
 \mathcal E (t) \leq 2 C_3\varepsilon , \quad  \forall \; t \in [0, T].
\end{equation}
Let $T^{*}$ be the largest possible time of $T$ for what \eqref{re} holds. Now, we only need to show $T^{*} = \infty$.  By the estimate of total energy \eqref{energy3}, we can use
 a standard continuation argument to show that $T^{*} = \infty$ provided that $\varepsilon$ is small enough.  We omit the details here. Hence, we finish the proof of Theorem \ref{th1}.

\section{The proof of the Theorem \ref{th2}}

In this section, we prove the second theorem of the present paper.
The local wellposedness of  \eqref{m1} with  $(u_0,\tau_0)\in H^\sigma(\R^2), (\sigma>4)$ can be obtained by using a   standard energy argument. Here, we omit it. To proved the global wellposedness, it suffices to obtain a global a priori bound for $\|(\u,\tau_\nu)\|_{H^\sigma}$. In fact, we only need to bound  $$\int_0^T\|(\nabla u_\nu,\nabla \tau_\nu)\|_{L^\infty}dt<\infty.$$
To do this, we will prove  $\|(\u,\tau_\nu)\|_{H^s}<\infty$ for any $0\le s\le \sigma-2$. By Theorem \ref{global}, one can deduce that $(u,\tau)\in C([0,T];H^\sigma(\R^2))$. Thus, in the following, we main concern with the bound of $\|(\u-u,\tau_\nu-\tau)\|_{H^s}.$
The proof in the following borrows some idea from \cite{wujiahong}.

Denote $\bar{u}=\u-u$, $\bar{\pi}=\pi_\nu-\pi$, and $\bar{\tau}=\tau_\nu-\tau$.
From \eqref{m1} and \eqref{m2}, we can deduce that $(\v,\w)$ satisfies the following equations:
\begin{eqnarray}\label{m3}
\left\{\begin{aligned}
&\partial_t \v+ u\cdot\nabla \v+\v\cdot\nabla (u+\v)-\nu \Delta (u+\v)+\nabla \bar{\pi}-  \div \bar{\tau}=0,\\
&\partial_t\bar{\tau} + u\cdot \nabla \bar{\tau}+ \v\cdot \nabla (\tau+\bar{\tau})  -\Delta\bar{\tau}-  D (\v)=0,\\
&\div \v =0,\\
&(\v,\w)|_{t=0}=(0,0).
\end{aligned}\right.
\end{eqnarray}

Denote $$J^{s}f\stackrel{\mathrm{def}}{=}\mathcal{F}^{-1}[(1+|\xi|^2)^{\frac{s}{2}}\hat{f}(\xi))].$$

Taking the $L^2$ inner product with $J^s\v,J^s\w$ to the first and second equation of \eqref{m3}, respectively, we have
\begin{align}\label{a1}
&\f12\frac{d}{dt}\|\v\|_{H^s}^2+\nu\|\nabla\v\|_{H^s}^2\nonumber\\
&\quad=\int_{\R^2}J^s(\div \bar{\tau})\cdot J^s\v \,dx-\nu\int_{\R^2}J^s(\Delta u)\cdot J^s\v \,dx\nonumber\\
&\quad=-\int_{\R^2}J^s(u\cdot\nabla \v)\cdot J^s\v \,dx-\int_{\R^2}J^s(\v\cdot\nabla u)\cdot J^s\v \,dx-\int_{\R^2}J^s(\v\cdot\nabla \v)\cdot J^s\v \,dx,
\end{align}
and
\begin{align}\label{a2}
\f12\frac{d}{dt}\|\w\|_{H^s}^2+\|\nabla\w\|_{H^s}^2
=&\int_{\R^2}J^s( D (\v))\cdot J^s\w \,dx
-\int_{\R^2}J^s(u\cdot\nabla \w)\cdot J^s\w \,dx
\nonumber\\
&-\int_{\R^2}J^s(\v\cdot\nabla \tau)\cdot J^s\w \,dx-\int_{\R^2}J^s(\v\cdot\nabla \w)\cdot J^s\w \,dx.
\end{align}

Summing up \eqref{a1} and \eqref{a2} and using the following cancelation,
$$
\int_{\R^2}J^s(\div \bar{\tau})\cdot J^s\v \,dx+\int_{\R^2}J^s( D (\v))\cdot J^s\w \,dx=0,
$$
 we get
\begin{align}\label{a3}
\f12\frac{d}{dt}\|(\v,\w)\|_{H^s}^2
\le &{\left|\nu\int_{\R^2}J^s(\Delta u)\cdot J^s\v \,dx\right|}\nonumber\\
&+{\left|\int_{\R^2}J^s(u\cdot\nabla \v)\cdot J^s\v \,dx-\int_{\R^2}J^s(u\cdot\nabla \w)\cdot J^s\w \,dx\right|}
\nonumber\\
&+{\left|\int_{\R^2}J^s(\v\cdot\nabla \v)\cdot J^s\v \,dx-\int_{\R^2}J^s(\v\cdot\nabla \w)\cdot J^s\w \,dx\right|}\nonumber\\
&+{\left|\int_{\R^2}J^s(\v\cdot\nabla u)\cdot J^s\v \,dx-\int_{\R^2}J^s(\v\cdot\nabla \tau)\cdot J^s\w \,dx\right|}.
\end{align}
Applying the H\"older inequality, we have
\begin{align}\label{a4}
\left|\nu\int_{\R^2}J^s(\Delta u)\cdot J^s\v \,dx\right|\le C\nu\|u\|_{H^{s+2}}\|\v\|_{H^{s}}.
\end{align}
Due to $\div u=0,$ one can use integrating by part to get
$$
\int_{\R^2}u\cdot\nabla J^s\v\cdot J^s\v \,dx=\int_{\R^2}u\cdot\nabla J^s\w\cdot J^s\w \,dx,$$
from which we can deduce that
\begin{align}\label{a5}
&\left|\int_{\R^2}J^s(u\cdot\nabla \v)\cdot J^s\v \,dx+\int_{\R^2}J^s(u\cdot\nabla \w)\cdot J^s\w \,dx\right|\nonumber\\
&\quad=\left|\int_{\R^2}(J^s(u\cdot\nabla \v)-u\cdot\nabla J^s\v)\cdot J^s\v \,dx+\int_{\R^2}(J^s(u\cdot\nabla \w)-u\cdot\nabla J^s\w)\cdot J^s\w \,dx\right|\nonumber\\
&\quad\le C(\|J^s(u\cdot\nabla \v)-u\cdot\nabla J^s\v\|_{L^2}\|J^s\v\|_{L^{2}}+\|J^s(u\cdot\nabla \w)-u\cdot\nabla J^s\w\|_{L^2}\|J^s\w\|_{L^{2}})\nonumber\\
&\quad\le C(\|\nabla u\|_{L^\infty}\|J^s\v\|_{L^2}+\|\nabla \v\|_{L^\infty}\|J^su\|_{L^2})\|J^s\v\|_{L^2}\nonumber\\
&\quad\quad+C(\|\nabla u\|_{L^\infty}\|J^s\w\|_{L^2}+\|\nabla \w\|_{L^\infty}\|J^su\|_{L^2})\|J^s\w\|_{L^2}\nonumber\\
&\quad\quad\le C\|u\|_{H^{s}}(\|\v\|_{H^{s}}^2+\|\w\|_{H^{s}}^2)
\end{align}
where we have used the following
 commutator estimate of Kato and Ponce \cite{ponce}:
\begin{align*}
\big\| J^{s} (u\cdot\nabla{v})-u\cdot\nabla J^{s}{v}\big\|_{L^2}\le C(\|\nabla u\|_{L^\infty}\|J^{s-1}\nabla v\|_{L^2}+\|\nabla v\|_{L^\infty}\|J^{s}u\|_{L^2}), \forall\, s\ge0 .
\end{align*}

Exact the same line as \eqref{a5}, one has
\begin{align}\label{a6}
\left|\int_{\R^2}J^s(\v\cdot\nabla \v)\cdot J^s\v \,dx+\int_{\R^2}J^s(\v\cdot\nabla \w)\cdot J^s\w \,dx\right|
\le C\|\v\|_{H^{s}}(\|\v\|_{H^{s}}^2+\|\w\|_{H^{s}}^2).
\end{align}
Using the fact that $H^s(\R^2)$ is an algebra for any $s>1$, thus, we can get
\begin{align}\label{a7}
 &\left|\int_{\R^2}J^s(\v\cdot\nabla u)\cdot J^s\v \,dx+\int_{\R^2}J^s(\v\cdot\nabla \tau)\cdot J^s\w \,dx\right|\nonumber\\
&\quad\le C(\|J^s(\v\cdot\nabla u)\|_{L^2}\|J^s\v\|_{L^2}+\|J^s(\v\cdot\nabla \tau)\|_{L^2}\|J^s\w\|_{L^2})\nonumber\\
&\quad\le C(\|\nabla u\|_{H^{s}}\|\v\|_{H^{s}}^2+\|\nabla\tau\|_{H^{s}}\|\v\|_{H^{s}}\|\w\|_{H^{s}})\nonumber\\
&\quad\le C(\| u\|_{H^{s+1}}\|\v\|_{H^{s}}^2+\|\tau\|_{H^{s+1}}\|\v\|_{H^{s}}\|\w\|_{H^{s}}).
\end{align}

Inserting \eqref{a4}--\eqref{a7} into \eqref{a3}  and letting
$$G(t)\stackrel{\mathrm{def}}{=}\|\v\|_{H^s}+\|\w\|_{H^s},\quad M(t)\stackrel{\mathrm{def}}{=}\|u\|_{H^{s+1}}+\|\tau\|_{H^{s+1}},$$
we can get
\begin{align*}
\f{d}{dt}G(t)\le \nu\|u(t)\|_{H^{s+2}}+C_1M(t)G(t)+C_2G^2(t).
\end{align*}

Multiplying by $\mathrm{e}^{-C_1\int_0^tM(t')dt'}$ on both hand side of the above inequality, we have
\begin{align}\label{a9}
\f{d}{dt}(G\,\mathrm{e}^{-C_1\int_0^tM(t')dt'})\le \nu\|u\|_{H^{s+2}}\,\mathrm{e}^{-C_1\int_0^tM(t')dt'}+C_2G^2\,\mathrm{e}^{-C_1\int_0^tM(t')dt'}.
\end{align}
By Lemma 1.3 of  \cite{constain}, if we set
\begin{align*}
\nu_0=\left(8C_2\int_0^T\|u(t')\|_{H^\sigma}\,\mathrm{e}^{C_1\int_{t'}^{T} M(\xi)d\xi}dt'\right)^{-1},
\end{align*}
then, for $0<\nu\le\nu_0$ and $0\le t\le T,$ we infer from \eqref{a9} that
\begin{align}\label{a11}
G(t)\le 12\nu\int_0^T\|u(t')\|_{H^\sigma}\,\mathrm{e}^{C_1\int_{t'}^{T} M(\xi)d\xi}dt'.
\end{align}
From \eqref{a11}, one can deduce that
 $$\|(\u,\tau_\nu)-(u,\tau)\|_{H^s}\le C(T)\nu,$$
where $ C(T)$ is  a constant dependent on $T$ and $\|(u,\tau)\|_{L^\infty([0,T];H^{\sigma})}$.

Consequently, we  have completed  the proof of Theorem \ref{th2}.

    \bigskip

\noindent \textbf{Acknowledgement.} {The first author of this work is supported by NSFC under the grant number 11601533,  and the
Science and Technology Program of Shenzhen under grant number 20200806104726001.
The third author is supported by NSFC under the grant number 11571118 and
the NSFC key project under the grant number 11831003.}

\end{document}